\documentclass[12pt]{article}
\usepackage{amsmath}
\usepackage{amssymb}
\usepackage{amscd}
\usepackage{amsthm}

\def\ZZ{{\mathbb Z}}
\def\RR{{\mathbb R}}

\def\QQ{{\mathbb Q}}
\def\PP{{\mathbb P}}
\def\CC{{\mathbb C}}
\def\cX{{\mathcal X}}
\def\cO{{\mathcal O}}
\def\bk{{\bf k}}

\newtheorem{lemma}{Lemma}[section]
\newtheorem{theorem}[lemma]{Theorem}
\newtheorem{corollary}[lemma]{Corollary}

\theoremstyle{definition}

\newtheorem{example}[lemma]{Example}

\newtheorem{remark}[lemma]{Remark}

\theoremstyle{remark}
\newtheorem*{proof*}{Proof}
\numberwithin{equation}{section}

\title{On the Grothendieck groups of toric stacks}
\author{Zheng Hua}

\begin{document}
\maketitle
\section{Introduction}
In this note, we prove that the Grothendieck group of a smooth
complete toric Deligne-Mumford stack is torsion free. This statement holds when the generic point is stacky. We also construct an example of open toric stack with torsion in K-theory. This is a part
of the author's Ph.D thesis. A similar result has been proved by Goldin, Harada, Holm,
Kimura and Knutson in \cite{symp} using symplectic methods.

\section{Grothendieck groups of reduced stacks}
Let $N$ be a free abelian group of rank $d$ and $N_\RR=N\otimes
\RR$. Given a complete simplicial fan $\Sigma$ in $N_\RR$, one
chooses an integral element $v_i$ in each of the one-dimensional
cones of $\Sigma$. This defines a stacky fan $\bf{\Sigma}$ in the
sense of $\cite{BCS}$. We denote the corresponding toric
Deligne-Mumford stack by $\cX_{\bf\Sigma}$. Recall the Grothendieck
group $K_0(\cX_{\bf\Sigma})$ is defined to be the free abelian group
generated by all formal combinations of coherent sheaves on
$\cX_{\bf\Sigma}$ modding out by the short exact sequences. Each
element $v_i$ corresponds to a toric invariant divisor $E_i$. This
divisor $E_i$ determines an invertible sheaf $\cO(E_i)$. We denote
its equivalent class in $K_0(\cX_{\bf\Sigma})$ by $R_i$. The ring
structure of $K_0(\cX_{\bf\Sigma})$ is given by tensor product of
coherent sheaves. K-theory of smooth toric stacks has been studied
in \cite{BH}. In particular they computed $K_0(\cX_{\bf\Sigma})$
explicitly by writing out its generators and relations.
\begin{theorem}\cite{BH}\label{bh}
Let $B$ be the quotient of the Laurent polynomial ring
$\ZZ[x_1,x_1^{-1},\ldots,x_n,x_n^{-1}]$ by the ideal generated by
the relations
\begin{itemize}
\item
$\prod_{i=1}^n x_i^{\langle m,v_i\rangle}=1$ for any dual vector
$m\in M=Hom(N,\ZZ)$,
\item
$\prod_{i\in S}(1-x_i)=0$ for any set $S\subseteq [1,\ldots,n]$ such
that $\{v_i|i\in S\}$ are not contained in any cone of $\Sigma$.
\end{itemize}
Then the map from $B$ to $K_0(\cX_{\bf\Sigma})$ which sends $x_i$ to
$R_i$ is an isomorphism of rings.
\end{theorem}
Our main result is the following.
\begin{theorem}\label{free}
The Grothendieck group $K_0(\cX_{\bf\Sigma})$ of a complete smooth
toric Deligne-Mumford stack $\cX_{\bf\Sigma}$ is a free $\ZZ$
module.
\end{theorem}
\begin{proof}

We denote the Laurent polynomial ring
$\ZZ[x_1,x_1^{-1},\ldots,x_n,x_n^{-1}]$ by $R$. Let $A=R/I$, where
$I$ is generated by $\prod_{i\in S}(1-x_i)=0$ for any set
$S\subseteq [1,\ldots,n]$ such that $\{v_i|i\in S\}$ are not
contained in any cone of $\Sigma$. And $B=A/J$, where $J$ is
generated by $n$ Laurent polynomials ${h_j=\prod_{i=1}^n
x_i^{\langle m_j,v_i\rangle}-1}$ where $m_j$ is an integral basis of
$M$.

First we want to replace $h_j$ by
$g_j=\prod_{<m_j,v_i>>0}x_i^{<m_j,v_i>}-\prod_{<m_j,v_i><0}x_i^{-<m_j,v_i>}$.
They generate the same ideal $J$ but this collection avoids negative
powers. To prove $B$ is a free $\ZZ$ module we need to show that the
multiplication map $B \rightarrow pB$ is an injection for any prime
$p$. Let $K(g_1,\ldots,g_d)$ and $K(g_1,\ldots,g_d,p)$ be the Koszul
complexes for sequences ${g_1,\ldots,g_d}$ and ${g_1,\ldots,g_d,p}$
of elements of the ring $A$. These two Koszul complexes are related
by the following lemma.

\begin{lemma}\cite{Eisenbud}
Let $\phi: K(g_1,\ldots,g_d)\rightarrow K(g_1,\ldots,g_d)$ be the
map of multiplication by $p$. Then $K(g_1,\ldots,g_d,p)$ equals
$Cone(\phi)[-1]$. Here Cone means mapping cone of complexes.
\end{lemma}
\begin{proof}
See corollary 17.11. of \cite{Eisenbud}.
\end{proof}

According to this lemma, we get a long exact sequence of cohomology
groups:
\begin{equation}\label{les}
\begin{CD}
\ldots @>>> H^i(K(g_1,\ldots,g_d,p)) @>>> H^i(K(g_1,\ldots,g_d))\\
@>{\phi}>>H^i(K(g_1,\ldots,g_d))@>>> H^{i+1}(K(g_1,\ldots,g_d,p))
@>>> \ldots
\end{CD}
\end{equation}

We will show that all the cohomology groups of $K(g_1,\ldots,g_d)$
and $K(g_1,\ldots,g_d,p)$ vanish except at one position. More
precisely, the only non vanishing piece of \eqref{les} is:
\[
\begin{CD}
0 @>>> H^n(K(g_1,\ldots,g_d,p))\cong B @>p>>
H^n(K(g_1,\ldots,g_d))\cong B \\@>>>
H^{n+1}(K(g_1,\ldots,g_d,p))\cong B/pB @>>> 0
\end{CD}
\]
To prove this we need a result about Cohen-Macaulay properties of
Stanley-Reisner rings.

\begin{theorem}\label{danilov}
Let $A'=\ZZ[x_1,\ldots,x_n]/I$. Ring $A'$ is Cohen-Macaulay.
\end{theorem}
\begin{proof}
If we make a change of variables $x_i$ to $1-x_i$, then we see that
$A'$ is nothing but the Stanley-Reisner ring associated to
supporting polytope of $\Sigma$. It is a general fact that the
Stanley-Reisner ring of polytopes are CM over any field(See Chapter
5 of \cite{CMR}). Furthermore one can show it is actually CM over
$\ZZ$(See Exercise 5.1.25 of \cite{CMR}). We will sketch the
solution of this exercise in the following remark.
\end{proof}

\begin{remark}
Consider the flat morphism $\ZZ\to A'$. For any maximal ideal
$\frak{q}\subset A'$, we have $\frak{q}\cap \ZZ=(p)$. In order to
show $A'$ is CM it suffices to check it for each fiber, i.e.
$A'_\frak{q}/p A'_\frak{q}$ is CM for all the maximal ideal
$\frak{q}$. If $(p)$ is not $(0)$ then $A'_\frak{q}/p
A'_\frak{q}=(A'\otimes \ZZ/p\ZZ)_\frak{q}$. It is CM because
Stanley-Reisner ring over the field is CM. So we just need to show
that for any maximal ideal $\frak{q}$, the restriction $\frak{q}\cap
\ZZ$ is not $(0)$. Suppose this is the case, we will have an
inclusion $\ZZ\to A'/\frak{q}$. However, since we assume $\frak{q}\cap
\ZZ=(0)$, the field $A'/\frak{q}$ must have characteristic zero. But
this contradicts the fact that $A'$ is finitely generated over $\ZZ$
because $\QQ$ is not finitely generated over $\ZZ$.
\end{remark}

\begin{corollary}
The ring $A$ is Cohen-Macaulay.
\end{corollary}
\begin{proof}
Because $A$ is a localization of $A'$ and being CM ring is a local
property, $A$ is CM by Theorem \ref{danilov}.
\end{proof}

\begin{remark}\label{dim}
It follows from the general theory of Stanley Reisner ring (Theorem
$5.1.16$ of \cite{CMR}) that $A'$ has Krull dimension $d+1$.
\end{remark}

\begin{lemma}\cite{Eisenbud}\label{main}
Suppose $M$ is a finitely generated module over ring $A$ and
$I=(x_1,\ldots,x_n)\subset A$ is a proper ideal. If $depth(I)=r$
then $H^i(M\bigotimes K(x_1,\ldots,x_n))=0$ for $i<r$, while
$H^r(M\bigotimes K(x_1,\ldots,x_n))=M/IM$.
\end{lemma}

\begin{lemma}\label{fd}
The quotient $A/J$ is a finitely generated abelian group.
\end{lemma}
\begin{proof}
Let $\bk$ be any field and $f$ be an arbitrary map from $A/J$ to
$\bk$. Maximal ideals of $A/J$ are in one to one correspondence with
such map $f$. We want to solve for $(x_1,\ldots,x_n)$ that satisfy
equations in ideal $I$ and $J$ in the field $\bk$. Recall elements
of ideal $I$ are in form of $\prod_{i\in S}(1-x_i)$ for any subset
$S \subseteq [1,\ldots,n]$ such that one dimensional rays $v_i,i\in
S$ are not contained in any cone of $\Sigma$. So $x_i$ equals 1
outside some cone $\sigma$. Then equations in $J$ reduce to
$\prod_{v_i\in\sigma}x_i^{\langle m, v_i\rangle}$=1. We can choose
the dual vector $m$ such that $\langle m, v_i\rangle=0$ for all but
one $i$. Say $\langle m, v_i\rangle=d_i$. The number $d_i$ only
depends on the fan but not on the field $\bk$. This implies that
$1-x_i^{d_i}$ maps to 0 for any map $f$ from $A/J$ to $\bk$, i.e.
$1-x_i^{d_i}$ is contained in any maximal ideal of $A/J$. Because
$A/J$ is a finitely generated $\ZZ$ algebra the Jacobson radical
coincides with nilradical. So $(1-x_i^{d_i})^N$=0 for any $i$. We
can pick a large enough integer $N$ uniformly for any $x_i$ such
that there exists a $\ZZ$ basis consisting of monomials with powers
of each $x_i$ between 0 and $Nd_i$. This proves the statement of the
lemma.
\end{proof}

By theorem \ref{danilov}, remark \ref{dim} and lemma \ref{fd} we can
prove:
\begin{corollary}
The ideal $J=(g_1,\ldots,g_d)$ has depth $d$.
\end{corollary}
\begin{proof}
Because $A$ is CM, by the definition of CM rings
$depth(J)=codim(J)$. The quotient $A/J$ is finitely generated over
$\ZZ$, therefore, of Krull dimension one. By remark \ref{dim}
$codim(J)=d$ and $depth(J)=d$.
\end{proof}

This corollary above together with lemma \ref{main} imply the Koszul
complex $K(g_1,\ldots,g_d)$ has only one nonzero cohomology
$H^d(K(g_1,\ldots,g_d))=B=A/J$. On the other hand, the lemma
\ref{fd} imples $B/pB$ is a finite dimensional vector space over
$\ZZ/p$. By similar argument with the corollary above we get
$depth(J,p)=d+1$. Then $H^i(K(g_1,\ldots,g_d,p))=B/pB$ when $i=d+1$
and zero otherwise. Now by applying the long exact sequence
\eqref{les} we prove the multiplication map by $p$ is an injection.
This finish the proof of theorem \ref{free}.
\end{proof}
\begin{remark}
The proof of theorem \ref{free} can be generalized to the non complete toric stacks satsifying a condition called ``shellability". This is a combinatorial condition on the underlying simplicial complex of the toric stack(See \cite{CMR} for details of this definition). It is proved in \cite{CMR} that Stanley-Reisner rings of shellable simplicial complexes are Cohen-Macaulay. However, we will see in Chapter \ref{open} that Grothendieck groups of open toric stacks are not necessarily free.
\end{remark}
\section{Grothendieck groups of non-reduced stacks}
Now we remove the assumption that $N$ is a free abelian group. Then
the corresponding toric stack will have nontrivial stabilizer at the
generic point. We will generalize theorem \ref{free} to this
setting. Recall the derived Gale dual of the homomorphism $\beta :
\ZZ^n\to N$ is the homomorphism $\beta^\vee : (\ZZ^n)^\vee\to
DG(\beta)$. When $N$ is torsion free, $DG(\beta)$ is the Picard
group. The general definition of $DG(\beta)$ involves a projective
resolution of $N$. We refer to \cite{BCS} for details. Theorem
\ref{bh} can be generalized to the case when $N$ has torsion. Notice
the ring $\ZZ[x_1,x_1^{-1},\ldots,x_n,x_n^{-1}]/J$ is the
representation ring of the algebraic group $Hom(DG(\beta),\CC^*)$
when $N$ is torsion free. If $N$ has torsion then
$Hom(DG(\beta),\CC^*)$ maps to $(\CC^*)^n$ with finite kernel. After
replacing $\ZZ[x_1,x_1^{-1},\ldots,x_n,x_n^{-1}]/J$ by the
representation ring of $Hom(DG(\beta),\CC^*)$ we can generalize
Theorem \ref{bh} to non reduced case(See section $6$ of \cite{BH} for
more details).

\begin{theorem}\label{gerb}
Let $N$ be a finitely generated abelian group and $\bf\Sigma$ is a
stacky fan in $N$. The Grothendieck Group $K_0(\cX_{\bf\Sigma})$ is
a free $\ZZ$ module.
\end{theorem}
\begin{proof}
Let's denote the $N_{free}$ for the quotient $N/torsion(N)$ and
$\cX_{red}$ for the reduced stack associated to $N_{free}$. Recall
the Grothendieck group $K_0(\cX_{\bf\Sigma})$ is the quotient of
representation ring of $Hom(DG(\beta),\CC^*)$ by the ideal $I$
generated by Stanley-Reisner relations. Let's denote the Gale dual
group of the reduced stack $\cX_{red}$ by $DG(\beta_{red})$. The
quotient map $N\to N_{free}$ induces an inclusion on Gale dual
groups $DG(\beta_{red})\to DG(\beta)$, whose cokernel is isomorphic
to $torsion(N)$. Now we see the Grothendieck groups
$K_0(\cX_{\bf\Sigma})$ and $K_0(\cX_{red})$ are isomorphic to the
group rings $\ZZ[DG(\beta)]$ and $\ZZ[DG(\beta_{red})]$. If we fix a
lifting from $torsion(N)$ to $DG(\beta)$, then we get a coset
decomposition $DG(\beta)=\sqcup_{y\in torsion(N)} (y
DG(\beta_{red}))$. This induce a coset decomposition of the group
ring $\ZZ[DG(\beta)]$ such that each coset is isomorphic with
$\ZZ[DG(\beta_{red})]$. Since $\ZZ[DG(\beta_{red})]$ is torsion free
by theorem \ref{free}, we prove the theorem.
\end{proof}

\section{Grothendieck groups of non complete stacks}\label{open}
Theorem \ref{bh} holds for non complete toric stacks too. But our
proof for freeness of K-theory relies on the shellability of the
underlying simplicial complex of the toric stack. There are many non
complete toric stacks whose underlying simplicial complexes are
\emph{not} shellable. For example, we can take $\PP^1\times\PP^1$.
Denote its four toric invariant divisors by $E_1,E_2,E_3$ and $E_4$.
Let point $P$(resp. $Q$) be the intersection of $E_1$ and
$E_2$(resp. $E_3$ and $E_4$). Simplicial complex of
$\PP^1\times\PP^1\backslash \{P,Q\}$ is not shellable.

Actually, there are examples of non complete toric stacks such that
their Grothendieck groups have torsions. The following example is
due to Lev Borisov.
\begin{example}
Let's take a dimension five weighted projective stack
$\PP(1,1,1,1,2,2)$. Denote its toric invariant divisors by
$E_1,E_2,\ldots,E_6$, where $E_1,\ldots,E_4$ have weights one and
$E_5,E_6$ have weights two. Let $\cX$ be the substack
$\PP(1,1,1,1,2,2)\backslash \{(E_1\cap E_2\cap E_3\cap E_4) \cup
(E_5\cap E_6)\}$. By theorem \ref{bh}
$$K_0(\cX)=\frac{\ZZ[t,t^{-1}]}{\langle(1-t)^4,(1-t^2)^2\rangle}$$
It is easy to check that $t(1-t)^2$ is a torsion element.
\end{example}

Department of Mathematics, University of Wisconsin-Madison,

Madison, WI, 53706, U.S.

hua@math.wisc.edu
\end{document}